\documentclass[a4paper,11pt]{article}
\usepackage[bookmarks=true,
              bookmarksnumbered=true, breaklinks=true,
              pdfstartview=FitH, hyperfigures=false,
              plainpages=false, naturalnames=true,
              colorlinks=true,
              pdfpagelabels]{hyperref}
\usepackage{geometry,latexsym,amssymb,amsmath,amsthm,color,bm}
\usepackage[latin5]{inputenc}
\usepackage{enumerate}
\usepackage{enumitem}
\usepackage[T1]{fontenc}
\usepackage{authblk}
\usepackage[all]{xy}
\usepackage{palatino}
\usepackage{indentfirst}
\usepackage{titlesec}
\usepackage{graphics}
\usepackage{tabularx}
\usepackage{lipsum}

\theoremstyle{plain}
\newtheorem{theorem}{Theorem}[section]
\newtheorem{proposition}[theorem]{Proposition}
\newtheorem{lemma}[theorem]{Lemma}
\newtheorem{corollary}[theorem]{Corollary}

\newenvironment{Prf}{{\bf Proof:} }{\hfill $\Box$\mbox{}}

\theoremstyle{definition}
\newtheorem{definition}[theorem]{Definition}
\newtheorem{example}[theorem]{Example}

\newtheorem{remark}[theorem]{Remark}

\def\Ker{\operatorname{Ker}}

\def\St{\mathsf{St}}
\def\Cost{\mathsf{Cost}}

\def\GpGd{\bm{\mathsf{GpGd}}}
\def\XMod{\bm{\mathsf{XMod}}}

\def\qed{\hfill $\Box$}

\def\inc{\operatorname{inc}}

\def\E{\operatorname{E}}

\def\XMod{\bm{\mathsf{XMod}}}
\def\GpGd{\bm{\mathsf{GpGd}}}
\def\Gp{\bm{\mathsf{Gp}}}

\def\XSq{\bm{\mathsf{X^2Mod}}}
\def\DGG{\bm{\mathsf{DbGpGd}}}

\def\GGdC/G{\bm{\mathsf{GpGpdCov/G}}}
\def\GGdCov/X{\bm{\mathsf{GpGpdCov/\pi X}}}
\def\GdC/G{\bm{\mathsf{GpdCov/G}}}
\def\GdA(G){\bm{\mathsf{GpdAct(G)}}}
\def\Act(G){\bm{\mathsf{GpdAct(G)}}}
\def\Cov/G{\bm{\mathsf{GpdCov/G}}}
\def\C{\bm{\mathsf{C}}}

\def\epsilon{\varepsilon}

\begin{document}
\title{Crossed modules, double group-groupoids and crossed squares}

\author[a]{Sedat TEMEL\thanks{S. Temel (e-mail : sdttml@gmail.com)}}
\author[b]{Tunçar ŞAHAN\thanks{T. Şahan (e-mail : tuncarsahan@gmail.com)}}
\author[c]{Osman MUCUK\thanks{O. Mucuk (e-mail : mucuk@erciyes.edu.tr)}}
\affil[a]{\small{Department of Mathematics, Recep Tayyip Erdoğan University, Rize, TURKEY}}
\affil[b]{\small{Department of Mathematics, Aksaray University, Aksaray, TURKEY}}
\affil[c]{\small{Department of Mathematics, Erciyes University, Kayseri, TURKEY}}

\date{}

\maketitle

\begin{abstract}

The purpose of this paper is to obtain the notion of crossed modules over group-groupoids considering split extensions; and prove a categorical equivalence between  these types of crossed modules and  double group-groupoids. This equivalence enables us to produce various examples of double groupoids.

\end{abstract}

\noindent{\bf Key Words:} Group-groupoid, crossed module, double group-groupoid, crossed square.
\\ {\bf Classification:} 20L05, 22A22, 18D35.

\section{Introduction}

In this paper we are interested in crossed modules of group-groupoids associated with the split extensions and  producing  new  examples of double groupoids in which the sets of squares, edges and points are group-groupoids.

The idea  of crossed module over groups  was initially introduced by Whitehead in \cite{Wth1} and \cite{Wth2} during the investigation of the properties of second relative homotopy groups for topological spaces. The categorical equivalence between  crossed modules over groups and group-groupoids which are widely called in literature \emph{2-groups} \cite{baez-lauda-2-groups}, {\em $\mathcal{G}$-groupoids} or \emph{group objects} in the category of groupoids  \cite{BS2},  was proved by Brown and Spencer in \cite[Theorem 1]{BS2}. Following this equivalence  normal and quotient objects in these two categories have been recently compared and associated objects in the category of group-groupoids have been characterized in \cite{Mu-Sa-Al}.   This categorical equivalence   has also been extended by Porter in \cite[Section 3]{Por} to a more general  algebraic category $\C$ called category of groups with operations whose idea goes back to Higgins \cite{Hig} and Orzech \cite{Orz,  Orz2}.  This result is used for example in  \cite{Ak-Al-Mu-Sa} as a  tool to extend some results about topological groups to the topological groups with operations.

Double groupoids which can be thought as a groupoid objects in the category of groupoids were introduced by Ehresmann in  \cite{Ehr,Ehr2}  and have been concerned by many mathematicians because of their connections with several branch of mathematics. For example according to  \cite [Chapter 6]{BHS}, the structure of crossed module is inadequate to give a proof of 2-dimensional Seifert-van-Kampen Theorem and hence one needs the idea of double groupoid. For the purpose of obtaining some examples of double groupoids,  Brown and Spencer in \cite{BS2}  proved the categorical equivalence between crossed modules over groups and  special double groupoids in the sense that the horizontal  and vertical groupoids agree and  the set of points is singleton. Then  the categorical equivalence of crossed modules over groupoids and  double groupoids with thin structures was proved in \cite[Chapter 6]{BHS}.)

% \cite{BS1} and  the concept of quotient double groupoid using congruences was also studied in terms of quotient groupoids in \cite{11} (see also \cite{SD}). The equivalence of  double groupoids and  crossed modules over groupoids  was given.

In this paper following Porter's methods in \cite{Por} we introduce the notion of crossed module over group-groupoids via  split extensions of short exact sequences and obtain the double groupoids associated with  these crossed modules. Moreover we have a categorical equivalence between these types of crossed modules and  double group-groupoids, which enables us to have  some varieties  of examples for double groupoids.

\section{Preliminaries}

A groupoid is defined to be a small category in which every arrow has an inverse \cite{Br1}.  More precisely by a {\em groupoid}  $G$ on $G_0$ we mean a set $G_{0}$ of {\em objects},  a set $G$ of arrows,  together with  {\em initial} and {\em final} point maps $d_0, d_1\colon G\rightarrow G_{0}$ and {\em object inclusion} map $\varepsilon  \colon G_{0} \rightarrow G$ such that $d_0\varepsilon=d_1\varepsilon=1_{G_{0}}$.  Moreover there is a  associative partial composition among the arrows with the property that if $a,b\in G$ provided $d_1(a)=d_0(b)$, then the {\em composite} $b\circ a$ exists such that $d_0(b\circ a)=d_0(a)$ and $d_1(b\circ a)=d_1(b)$. For $x\in G_{0}$ the element $\varepsilon (x)$ denoted by $1_x$ acts as the identity and each arrow $a$ has an inverse $a^{-1}$ such that $d_0({a}^{-1})=d_1(a)$, $d_1({a}^{-1})=d_0(a)$, $a\circ {a}^{-1} =(\varepsilon d_1)(a)$ and $ {a}^{-1}\circ a=(\varepsilon d_0)(a)$. We denote such a groupoid  only by $G$. A groupoid $G$ with only identity arrows is said to be  \textit{discrete groupoid}.

In a groupoid $G$ the set $d_{0}^{-1}(x)$ for an object $x$ in $G_0$ is called the \textit{star} of $G$ at $x$ and denoted by $\St_Gx$. Similarly the set $d_{1}^{-1}(x)$ is called the \textit{costar} of $G$ at $x$ and denoted by $\Cost_Gx$.

A {\em group-groupoid} $G$ is an internal groupoid  in the category of groups, i.e.,  $G$ and $G_0$ are groups with the property that  the initial and final point maps, object inclusion map,  the inversion and partial composition are group morphisms. An alternative name used in literature  for a group-groupoid is 2-group \cite{baez-lauda-2-groups}.  In a group-groupoid the group operation is written additively while  the composition in the groupoid by $`\circ'$ as above.  Recently, group-groupoid aspect of the  monodromy groupoid was developed in  \cite{Mu-Be-Tu-Na}.

We recall  that a crossed module over groups originally defined by  Whitehead \cite{Wth1,Wth2}, consists of two groups $A$ and $B$, an action of $B$ on $A$ denoted by $b\cdot a$ for $a\in A$ and $b\in B$; and a morphism  $\partial\colon A\rightarrow B$ of groups such that $\partial(b\cdot a)=b+\partial(a)-b$ and $\partial(a)\cdot a_1=a+a_1-a$ for all $a,a_1\in A$ and $b\in B$. We  denote such a crossed module by $(A,B,\partial)$. A morphism  $(f_1,f_2)$ from $(A,B,\partial)$ to $(A',B',\partial')$  is a pair of morphisms of groups $f_1\colon A\rightarrow A'$ and $f_2\colon B\rightarrow B'$  such that $f_2\partial=\partial'f_1$ and  $f_1(b\cdot a)=f_2(b)\cdot f_1(a)$ for  $a\in A$ and $b\in B$.

%Hence we have a category $\XMod(\Gp)$ of crossed modules over groups.

%It was proved by Brown and Spencer in \cite[Theorem 1]{BS1} that the category $\XMod(\Gp)$ of crossed modules over groups is equivalent to the category $\GpGd$ of group-groupoids. In \cite{Por} Porter proved a similar result  holds for  a certain algebraic category introduced by Orzech \cite{Orz} and  called category of groups with operations. Using the result of Brown and Spencer \cite{BS1}, in \cite{Mu-Sa-Al} the authors obtained the notion of normal and quotient objects in the category of group-groupoids.  In the latter paper if $G$ is a group-groupoid, $N$ is a subgroupoid of $G$ and $N_1$ is a normal subgroup of $N_1$ then $N$ is called a {\it normal subgroup-groupoid} of $G$.

A {\em  double groupoid} denoted by $\mathcal{G}=(S,H,V, P)$ has the sets $S$, $H$, $V$ and  $P$ of squares, horizontal edges, vertical edges and  points respectively. The set $S$ of the squares has groupoid structures on $H$ on $V$ which are also groupoids on $P$ and the groupoid structures are compatible with each other. In a double groupoid a square $u$ has bounding edges as follows

\[\begin{minipage}{0.5\textwidth}
\begin{xy}
*=<2cm,2cm>\txt{$u$}*\frm{-};
(14,0) *{d_1^hu} ;
(-14,0) *{d_0^hu} ;
(0,13) *{d_0^vu} ;
(0,-13) *{d_1^vu} ;
\end{xy}
\end{minipage}
\begin{minipage}{0.5\textwidth}
\begin{displaymath}
\xymatrix{
\ar[d]  \ar[r] & \circ_h \\
\circ_v & }
\end{displaymath}
\end{minipage}\]
 and the horizontal and vertical compositions of the squares are denoted by $v \circ_h u$ and $v\circ_v u$.

 In particularly if the horizontal and vertical groupoids coincide, then it is said to be a {\it special double groupoid}. According to Brown and Spencer \cite{BS2}  crossed modules over groups and special double groupoids where the set of points is singleton are categorically equivalent. By the detail of the proof for given a crossed module $\partial \colon A\rightarrow B$ there is a special double groupoid  $\mathcal{G}$  in which   the set $S$ of squares consists of the elements
\[ \left(\alpha ~;~\left( a ~\begin{array}{c} c \\   b \\ \end{array} ~ d\right)\right)\]
for  $ a, b, c, d\in B$ with  $\partial(\alpha)=b^{-1}a^{-1}cd$.

The horizontal and vertical compositions of squares are respectively defined to be \[\left(\beta ~;~ c ~\begin{array}{c} f \\   h \\ \end{array} ~ g\right)\circ_h \left(\alpha ~;~ a ~\begin{array}{c} b \\   d \\ \end{array} ~ c\right)=
\left(\alpha \beta^{d^{-1}} ~;~ a ~\begin{array}{c} bf \\   dh \\ \end{array} ~ g\right)\]

\[\left(\tau ~;~ j ~\begin{array}{c} d \\   i \\ \end{array} ~ h\right)\circ_v \left(\alpha ~;~ a ~\begin{array}{c} b \\   d \\ \end{array} ~ c\right) =\left(\alpha^{j}\tau ~;~ aj ~\begin{array}{c} b \\   i \\ \end{array} ~ ch\right) .\]
 See  \cite [Chapter 6]{BHS} for more discussions on double groupoids in which horizontal and vertical edges are same.

\section{Extensions and crossed modules of group-groupoids}

The idea  of groups with operations goes back to  \cite{Hig} and  \cite {Orz} (see also \cite{Orz2})
and it is adapted in  \cite{Por} and  \cite[p.21]{Tamar}  as follows:

A category $\C$  of groups with a set  of  operations $\Omega$ and with a set $\E$  of identities such that $\E$ includes the group laws, and the following conditions hold for the set $\Omega_i$  of $i$-ary operations in $\Omega$ is said to be a {\em category of groups with operations}:

(a) $\Omega=\Omega_0\cup\Omega_1\cup\Omega_2$;

(b) The group operations written additively $0,-$ and $+$ are
the  elements of $\Omega_0$, $\Omega_1$ and
$\Omega_2$ respectively. Let $\Omega_2'=\Omega_2\backslash \{+\}$,
$\Omega_1'=\Omega_1\backslash \{-\}$ and assume that if $\star\in
\Omega_2'$, then $\star^{\circ}$ defined by
$a\star^{\circ}b=b\star a$ is also in $\Omega_2'$. Also assume
that $\Omega_0=\{0\}$;

(c) For each   $\star \in \Omega_2'$, $\E$ includes the identity
$a\star (b+c)=a\star b+a\star c$;

(d) For each  $\omega\in \Omega_1'$ and $\star\in \Omega_2' $, $\E$
includes the identities  $\omega(a+b)=\omega(a)+\omega(b)$ and
$\omega(a)\star b=\omega(a\star b)$.

Topological version of this definition was given in \cite{Mu-Tu} and monodromy groupoids of internal groupoids within topological groups with operations were investigated in \cite{Mu-Ak}.

According to  \cite{Por} for groups with operations $A$ and $B$  an {\em extension} of $A$ by $B$ is
an exact sequence
\[\xymatrix{
	\bm{\mathsf{0}} \ar[r] &   A \ar@{->}[r]^{\imath} &   E \ar@{->}[r]^{p} &   B %\ar@/_/[l]_{s}
	\ar[r] & \bm{\mathsf{0}} }\]
in which $p$ is surjective and $\imath$ is the kernel of $p$.  It is {\em split } if there exists a morphism $s\colon  B \to E$ such that $p s = 1_B$.  For given  such a split extension  an action of $B$ on $A$  called {\em derived action} which is due to  Orzech \cite[p.293]{Orz} is defined by
\begin{equation} \label{eq12} \begin{array}{rcl}
b\cdot a & = &s(b)+a-s(b)\\
 b\star a  & = &s(b)\star a.
   \end{array} \tag{3.1}\end{equation}
for $b\in B$, $a\in A$ and $\star\in \Omega'_2$.

In the rest of this section applying the methods of  \cite{Por}  to the group-groupoids we obtain the notion of crossed modules for them: Let $G$ and $  H$ be two group-groupoids. We define an {\em extension} of $  H$ by $  G$ to be  a short exact sequence of group-groupoids
   \[\xymatrix{
	\mathcal{E} : & \bm{\mathsf{1}} \ar[r] &   G \ar@{->}[r]^{\iota} &   K \ar@{->}[r]^{p} &   H %\ar@/_/[l]_{s}
	\ar[r] & \bm{\mathsf{1}} } \]
where $\bm{\mathsf{1}}$ represents a singleton group-groupoid. Hence $  G=\Ker p$ and   $p$ is an epimorphism; and therefore  $ G$ can be considered as a normal subgroup-groupoid of $ K$. For given such an extension  we have the  group extensions
\[\xymatrix{
	\mathcal{E}_1 : & 0 \ar[r] & G \ar@{->}[r]^{\iota_{1}} & K \ar@{->}[r]^{p_{1}} & H %\ar@/_/[l]_{s}
	\ar[r] & 0  }\]
\[\xymatrix{
	\mathcal{E}_0 : & 0 \ar[r] & G_0 \ar@{->}[r]^{\iota_{0}} & K_0 \ar@{->}[r]^{p_{0}} & H %\ar@/_/[l]_{s}
	\ar[r] & 0  }\]
along with the morphisms of them
\[\xymatrix{
	\mathcal{E}_1 : & 0 \ar[r] & G \ar@{->}[r]^{\iota_{1}} \ar@<.8ex>[d]|-{d_{0}}\ar@<-.8ex>[d]|-{d_{1}} & K \ar@{->}[r]^{p_{1}} \ar@<.8ex>[d]|-{d_{0}}\ar@<-.8ex>[d]|-{d_{1}} & H \ar@<.8ex>[d]|-{d_{0}}\ar@<-.8ex>[d]|-{d_{1}} %\ar@/_/[l]_{s}
	\ar[r] & 0 \\
	\mathcal{E}_0: & 0 \ar[r] & G_0 \ar@{->}[r]_{\iota_{0}} \ar@/^1pc/[u]^{\varepsilon} & K_0 \ar@{->}[r]_{p_{0}} \ar@/^1pc/[u]^{\varepsilon} & H \ar@/^1pc/[u]^{\varepsilon} %\ar@/_/[l]_{s}
	\ar[r] & 0  }\]
Hence a group-groupoid extension $\mathcal{E}$ can be thought as an internal  groupoid in the category of group extensions.

 Replacing groups with operations in \cite{Por}  with group-groupoids we define  an   extension to be   {\it split} if there exists a morphism $s\colon   H\rightarrow   K$ of group-groupoids such that $ps=1_{  H}$
\[\xymatrix{
	\mathcal{E} : & \bm{\mathsf{1}} \ar[r] &   G \ar@{->}[r]^{\iota} &   K \ar@{->}[r]_{p} &   H \ar@/_/[l]_{s} \ar[r] & \bm{\mathsf{1}} }\]
In this case both extensions of groups $\mathcal{E}_1$ and $\mathcal{E}_0$ above become split.

We now obtain semidirect product  of group-groupoids as follows:
Let $\mathcal{E}$ be  a split extension  of $  H$ by $  G$. Then the  functor $\theta\colon   G\times   H\rightarrow   K$ defined  by $\theta(a,b)=a+s(b)$ on arrows has an inverse $\theta^{-1}(c)=(c-sp(c),p(c))$ for $c\in K$; and  $  G\times   H$ is a  group-groupoid with the group addition defined by
\begin{align*}
(a,b)+(a_1,b_1) & = \theta^{-1}\left(\theta\left((a,b)+(a_1,b_1)\right)\right) \\
& = \theta^{-1}\left(\theta(a,b)+\theta(a_1,b_1)\right) \\
& = \theta^{-1}\left(a+s(b)+a_1+s(b_1)\right) \\
& = \left(a+(s(b)+a_1-s(b)),b+b_1\right)
\end{align*}
for all $a,a_1\in G$ and $b,b_1\in H$.  Moreover  $\theta$ is an isomorphism of  group-groupoids.  Hence $G\times H$ becomes a group-groupoid inherited by $H$. We call the group-groupoid  $  G\times   H$ obtained  {\it semidirect product} of group-groupoids $ G$ and $H$ and  denote by $  G\rtimes   H$. So we have a split extension
\[\xymatrix{
	\mathcal{E}_{  G\rtimes   H} : & \bm{\mathsf{1}} \ar[r] &   G \ar@{->}[r]^-{\iota} &   G\rtimes   H \ar@{->}[r]_-{p} &   H \ar@/_/[l]_-{s} \ar[r] & \bm{\mathsf{1}} }\]
associated with $G\rtimes H$.
\begin{lemma} For a split extension  $\mathcal{E}$ of $H$ by $G$, the group $H$  acts on the group $G$.
\end{lemma}
\begin{Prf} The action of $H$ on $G$  is defined by \[b\cdot a=s(b)+a-s(b)\]
for $a\in G, b\in H$.\end{Prf}

We call  such an action  {\em derived action} of group-groupoids.

\begin{remark} \label{remark} If there exists a derived action of  $H$  on $G$, then for  $a,a_1\in G$ and  $b,b_1\in H$ the following hold.
\begin{enumerate}
 \item $d_0(b\cdot a)=d_0(b)\cdot d_0(a)$;
 \item  $d_1(b\cdot a)=d_1(b)\cdot d_1(a)$;
 \item\label{rem3}  $1_y\cdot 1_x=1_{y\cdot x}$ for $x\in G_0$ and  $y\in H_0$;
 \item $(b\cdot a)^{-1}=b^{-1}\cdot a^{-1}$;
 \item \label{interchang} $(b\circ b_1)\cdot (a\circ a_1)=(b\cdot a)\circ(b_1\cdot a_1)$ whenever one side makes sense.
   \end{enumerate}
 \end{remark}

\begin{lemma} For group-groupoids $G$ and $H$, if the group $H$ acts on $G$, then $H_0$ also acts on $G_0$ by an action \[y\cdot x=d_0(\varepsilon(y\cdot x))=d_0(\varepsilon(y)\cdot\varepsilon(x))\]
for all $y\in H_0$ and $x\in G_0$. 
\end{lemma}
\begin{proof}
It is a consequence of Remark \ref{remark} (\ref{rem3}).
\end{proof}

Hence we can define  group-groupoid action as follows:

\begin{definition}
	Let $  G$ and $  H$ be two group-groupoids. If there is a group action of $H$ on $G$, then  we say that $  H$ acts on $  G$.
\end{definition}

\begin{lemma}
An action of $  H$ on $  G$ is a derived action if and only if $ G\rtimes   H$ is a group-groupoid with the group addition given by
\begin{align*}
(a,b)+(a_1,b_1)&= \left(a+b\cdot a_1,b+b_1\right)
\end{align*}
for all $a,a_1\in G$, $b,b_1\in H$. 
\end{lemma}

As an assessment of  Remark \ref{remark} (\ref{interchang}) the proof of the following lemma is immediate.

\begin{lemma}\label{indact} For  a derived action of $  H$ on $  G$, the following hold:

\begin{enumerate}

\item $\Ker d_{0}^{  H}$ acts on $\Ker d_{1}^{  G}$  by $b\cdot a_1=a_1$  for  $a_1\in \Ker d_{1}^{  G}$ and
$b\in \Ker d_{0}^{  H}$.
\item $\Ker d_{1}^{  H}$  acts on $\Ker d_{0}^{  G}$ by $b_1\cdot a=a$ for $a\in \Ker d_{0}^{  G}$   and $b_1\in \Ker d_{1}^{  H}$.
\item  $b\cdot a = \varepsilon^{  H}d_{1}^{  H}(b)\cdot a $ and $ b\cdot a = b\cdot \varepsilon^{G}d_{1}^{  G}(a)-\varepsilon^{G}d_{1}^{  G}(a)+a$.
\end{enumerate}

\end{lemma}

\begin{example} For any group-groupoid $G$, the conjugation action of $  G$ on itself given by
$a\cdot a_1=a+a_1-a $ for all $a,a_1\in G$ is a derived action. Hence there is a split extension of $  G$ by $  G$
\[\xymatrix{
\mathcal{E}_{  G} : & \bm{\mathsf{1}} \ar[r] &   G \ar@{->}[r]^-{\iota} &   G\rtimes   G \ar@{->}[r]_-{p} &   G \ar@/_/[l]_-{s} \ar[r] & \bm{\mathsf{1}} }\]
 associated with  the conjugation action, where $\iota(a)=(a,0)$, $p(a,a_1)=a_1$, $s(a)=(0,a)$  for all $a,a_1\in G$.
 \end{example}

\begin{definition} Let
\[\xymatrix{
	\mathcal{E} : & \bm{\mathsf{1}} \ar[r] &   G \ar@{->}[r]^{\iota} &   K \ar@{->}[r]_{p} &   H \ar@/_/[l]_{s} \ar[r] & \bm{\mathsf{1}} }\]
and
\[\xymatrix{
	\mathcal{E}' : & \bm{\mathsf{1}} \ar[r] &   G' \ar@{->}[r]^{\iota} &   K' \ar@{->}[r]_{p} &   H' \ar@/_/[l]_{s} \ar[r] & \bm{\mathsf{1}} }\]
be two split extensions of group-groupoids. We define a morphism $(\alpha,\beta,\gamma)\colon\mathcal{E}\rightarrow \mathcal{E}'$ of split extensions to be consisting of  morphisms $\alpha\colon   G\rightarrow   G'$, $\beta\colon   K\rightarrow   K'$ and $\gamma\colon   H\rightarrow   H'$ such that the following diagram commutes.
\[\xymatrix{
	\mathcal{E} : \ar@<-.8ex>[d]_{(\alpha,\beta,\gamma)} & \bm{\mathsf{1}} \ar[r] &   G \ar@{->}[r]^{\iota}\ar@{->}[d]_{\alpha}  &   K \ar@{->}[r]_{p} \ar@{->}[d]_{\beta} &   H \ar@/_/[l]_{s} \ar[r] \ar@{->}[d]^{\gamma} & \bm{\mathsf{1}} \\
	\mathcal{E}' : & \bm{\mathsf{1}} \ar[r] &   G' \ar@{->}[r]^{\iota'} &   K' \ar@{->}[r]_{p'} &   H' \ar@/_/[l]_{s'} \ar[r] & \bm{\mathsf{1}} }\]
\end{definition}

\begin{remark} If there is a derived action of $H$ on $G$, then the morphism  $(1_{  G},\theta,1_{  H})\colon \mathcal{E}_{  G\rtimes  H}\rightarrow \mathcal{E}$ of  split extensions denoted below is an isomorphism of split extensions.

\[\xymatrix{
	\mathcal{E}_{  G\rtimes   H} : \ar@<-.8ex>[d]_{(1_{  G},\theta,1_{  H})} & \bm{\mathsf{1}} \ar[r] &   G \ar@{->}[r]^-{\iota} \ar@{->}[d]_{1_G} &   G\rtimes   H \ar@{->}[r]_-{p} \ar@{->}[d]_{\theta} &   H \ar@{->}[d]^{1_{  H}} \ar@/_/[l]_-{s} \ar[r] & \bm{\mathsf{1}}\\
	\mathcal{E} :  & \bm{\mathsf{1}} \ar[r] &   G \ar@{->}[r]^-{\iota}  &   K \ar@{->}[r]_-{p}  &   H \ar@/_/[l]_-{s} \ar[r]  & \bm{\mathsf{1}} }\]
\end{remark}

Following the idea in \cite{Por} we now define crossed module of group-groupoids  as follows:
\begin{definition}Let $  G$ and $  H$ be two group-groupoids with an action of $  H$ on $  G$. We call a morphism $\partial\colon   G\rightarrow   H$ of group-groupoids  {\it crossed module} of them whenever  \[(1_{  G},1_{  G}\times\partial, \partial)\colon \mathcal{E}_{  G}\rightarrow \mathcal{E}_{  G\rtimes  H} \text{\ \ \ \ and\ \ \ \ } (\partial, \partial\times 1_{  H},1_{  H})\colon \mathcal{E}_{  G\rtimes  H}\rightarrow \mathcal{E}_{  H}\]
denoted below are morphisms of split extensions.
\[\xymatrix{
	\mathcal{E}_{  G} : \ar@<-.8ex>[d]_{(1_{  G},1_{  G}\times\partial, \partial)} & \bm{\mathsf{1}} \ar[r] &   G \ar@{->}[r]^-{\iota}\ar@{->}[d]_{1_{  G}}  &   G\rtimes   G \ar@{->}[r]_-{p} \ar@{->}[d]_{1_{  G}\times\partial} &   G \ar@/_/[l]_-{s} \ar[r] \ar@{->}[d]^{\partial} & \bm{\mathsf{1}} \\
	\mathcal{E}_{  G\rtimes   H} : \ar@<-.8ex>[d]_{(\partial, \partial\times 1_{  H},1_{  H})} & \bm{\mathsf{1}} \ar[r] &   G \ar@{->}[r]^-{\iota} \ar@{->}[d]_{\partial} &   G\rtimes   H \ar@{->}[r]_-{p} \ar@{->}[d]_{\partial\times 1_{  H}} &   H \ar@{->}[d]^{1_{  H}} \ar@/_/[l]_-{s} \ar[r] & \bm{\mathsf{1}}\\
	\mathcal{E}_{  H} :  & \bm{\mathsf{1}} \ar[r] &   H \ar@{->}[r]^-{\iota}  &   H\rtimes   H \ar@{->}[r]_-{p}  &   H \ar@/_/[l]_-{s} \ar[r]  & \bm{\mathsf{1}} }\]
\end{definition}

We write $(G,H,\partial)$ for such a  crossed module. By the assessment of  the above morphisms of split extensions we can state the crossed module over group-groupoids as follows:

\begin{proposition} Let $\partial\colon G\rightarrow H$ be a morphism of group-groupoids such that  $H$ acts on $  G$. Then $(G,  H,\partial)$ is a crossed module over group-groupoids if and only  if $(G,  H,\partial_1)$  is a crossed module over groups.\end{proposition}

We should point out that  in a crossed module $(G,  H,\partial)$ over group-groupoids,  $\partial_0\colon G_0\rightarrow H_0$ is also a crossed  module of groups.

\begin{lemma}
	Let $(  G,  H,\partial)$ be a crossed module over group-groupoids. Then  $H_0$ acts  on $G$ by \[y\cdot a=\varepsilon^{  H}(y)\cdot a=1_y\cdot a\] for $a\in G, y\in H_0$  and $H$ acts on $G_0$ by  \[b\cdot x=d_{1}^{  H}(b)\cdot x\] for all  $b\in H, x\in G_0$ .
\end{lemma}

\begin{proof}
These can be seen by easy calculations. So proof is omitted.
\end{proof}

We now give the following examples of crossed modules for group-groupoids:

\begin{example}
Let $  G$ be a group-groupoid and $  N$ a normal subgroup-groupoid of $  G$ in the sense of \cite{Mu-Sa-Al}, i.e., $H$ is a normal subgroup of $N$ and hence $N_0$ is a subgroup of $G_0$. Then $(  N,  G,\inc)$ is a crossed module over group-groupoids where $\inc\colon   N\hookrightarrow  G$ is the inclusion functor and the action of $G$ on $N$ is conjugation. In particularly $(  G,  G,1_{  G})$ and $(\bm 1,   G,0)$ are crossed modules over group-groupoids.
\end{example}

\begin{example}
Let $(A,B,\partial)$ be a crossed module over groups. Then $A\rtimes A$ and $B\rtimes B$ are group-groupoids respectively on $A$ and $B$; and then  $(A\rtimes A, B\rtimes B, \alpha\times\alpha)$ becomes  a crossed module over group-groupoids.
\end{example}

To give a geometric example of crossed modules over group-groupoids we first recall that by a {\em topological crossed module} we a mean  a crossed module $(A,B,\partial)$ in which $A$ and $B$ are topological groups, the action of  $B$ on $A$ is continuous and $\partial$ is a continuous morphism of topological groups.
\begin{example}
It is known from \cite{BS1} that  if $X$ is a topological group,  then the fundamental groupoid $\pi X$ is a group-groupoid . Therefore if $(A,B,\partial)$ is a   topological crossed module, then  $(\pi A,\pi B,\pi(\partial))$ becomes a crossed module of  group-groupoids.
\end{example}

\begin{example}
A group can be thought as a discrete group-groupoid in which the arrows are only identities.  Hence every crossed module over groups $(A,B,\partial)$ can be considered as a crossed module over group-groupoids.
\end{example}

\begin{example}
For a group $A$, the  direct product $G=A\times A$ becomes a group-groupoid  on $A$ with $d_0(a,b)=a$, $d_1(a,b)=b$, $\varepsilon(a)=(a,a)$, $n(a,b)=(b,a)$ and $(b,c)\circ(a,b)=(a,c)$. Hence a crossed module over groups $(A,B,\partial)$ gives rise to a crossed module over group-groupoids replacing $A$ and $B$ with the associated group-groupoids.
\end{example}

A morphism  $(f,g)\colon (G,  H,\partial)\rightarrow (  G',  H',\partial')$ between crossed modules over group-groupoids is defined to be a pair of group-groupoid morphisms $f\colon  G\rightarrow  G'$ and $g\colon  H\rightarrow  H'$ with the property that  $(f,g)\colon(G,H,\partial)\rightarrow(G',H',\partial')$ is a morphisms of crossed modules over groups.  We  write $\XMod(\GpGd)$ for the category of crossed modules over group-groupoids and morphisms between them to be arrows.

\section{Crossed modules and double group-groupoids}

In this section we define a double group-groupoid to be an internal groupoid in the category of group-groupoids; and then prove that these are categorically equivalent to associated crossed modules.

%In this section we will show well-known categories which are equivalent to the category $\XMod(\GpGd)$ of crossed modules over group-groupoids such as the category $\XSq(\Gp)$ of crossed squares over groups (equivalently, crossed modules of crossed modules over groups), the category $\DGG$ of double group-groupoids (double groupoid objects in the category of groups or groupoid objects in the category of group-groupoids).

%\[\XMod(\XMod(\Gp)) \cong \XSq(\Gp) \cong \XMod(\GpGd) \cong \Cat(\GpGd) \cong \DGG\]

% \subsection{$\XMod(\GpGd) \cong \DGG$}

If $G$ is an internal groupoid in the category of group-groupoids, then the following structural groupoid maps are morphisms of group-groupoids provided $G_0=H$
 \[\xymatrix{
	 G {_{d_{0}}\times_{d_{1}}}  G \ar[r]^-m &  G \ar@(dl,dr)[]|{ n } \ar@<.7ex>[r]^-{d_{0}} \ar@<-.4ex>[r]_-{d_{1}}  &  H  \ar@(u,u)[l]_-{\varepsilon}}\]
Here we have four different but compatible group-groupoid structures
$(G,G_0)$, $(H,H_0)$, $(G,H)$ and $(G_0,H_0)$.

\begin{center}
	\begin{tabular}{m{0.05\linewidth}m{0.15\linewidth}}
		\centering $\mathcal{G}$ : &
		\centering $\xymatrix @=4pc {
			G \ar@<1ex>[r]^{d_{0}} \ar@<-1ex>[r]_{d_{1}} \ar @<1ex>[d]^{d_1}  \ar@<-1ex>[d]_{d_0} & H   \ar[l]|{\varepsilon_1}  \ar @<1ex>[d]^{d_1} \ar @<-1ex> [d]_{d_0} \\
			G_0 \ar[u]|{\varepsilon}  \ar@<1ex>[r]^{d_{0}} \ar@<-1ex>[r]_{d_{1}} & H_0 \ar[l]|{\varepsilon_0} \ar[u]|{\varepsilon}}$
	\end{tabular}
\end{center}

 Hence  we define a {\em double group-groupoid} to be consisting of four different, but  compatible, group-groupoids $(S,H)$, $(S,V)$, $(H,P)$ and $(V,P)$ such that the following diagram of group-groupoids commutes

\begin{center}
\begin{tabular}{m{0.05\linewidth}m{0.15\linewidth}}
\centering $\mathcal{G}$ : &
\centering $\xymatrix @=4pc {
S \ar@<1ex>[r]^{d_{0}^{h}} \ar@<-1ex>[r]_{d_{1}^{h}} \ar @<1ex>[d]^{d_1^{v}}  \ar@<-1ex>[d]_{d_0^{v}} & H   \ar[l]|{\varepsilon^h}  \ar @<1ex>[d]^{d_1^{{H}}} \ar @<-1ex> [d]_{d_0^{ H}} \\
V \ar[u]|{\varepsilon^{v}}  \ar@<1ex>[r]^{d_{0}^{V}} \ar@<-1ex>[r]_{d_{1}^{V}} & P \ar[l]|{\varepsilon^V} \ar[u]|{\varepsilon^{ H}}}$
\end{tabular}
\end{center}

Horizontal and vertical compositions together with group operations  have the following interchange laws:
\begin{align*}
(\beta\circ_{v}\alpha)\circ_{h}(\beta_1\circ_{v}\alpha_1)&= (\beta\circ_{h}\beta_1)\circ_{v}(\alpha\circ_{h}\alpha_1)\\
(\beta\circ_{v}\alpha)+(\beta_1\circ_{v}\alpha_1)&= (\beta+\beta_1)\circ_{v}(\alpha+\alpha_1) \ \ \tag{4.1}\label{inthv}\\
(\beta\circ_{h}\alpha)+(\beta_1\circ_{h}\alpha_1)&= (\beta+\beta_1)\circ_{h}(\alpha+\alpha_1)
\end{align*}
whenever one side of the equations make sense.

We can now give the following examples of double group-groupoids:
\begin{example} If $  G$ is a group-groupoid, then $\mathcal{G}=(G,G,G_0,G_0)$ is a double group-groupoid with the trivial structural maps
\begin{center}
	\begin{tabular}{m{0.05\linewidth}m{0.15\linewidth}}
		\centering $\mathcal{G}$ : &
		\centering $\xymatrix @=4pc {
			G \ar@<1ex>[r]^{1} \ar@<-1ex>[r]_{1} \ar @<1.2ex>[d]^{d_1^{G}}  \ar@<-1.2ex>[d]_{d_0^{G}} & G   \ar[l]|{1}  \ar @<1.2ex>[d]^{d_1^{{G}}} \ar @<-1.2ex> [d]_{d_0^{G}} \\
			G_0\ar[u]|{\varepsilon^{G}}  \ar@<1ex>[r]^{1} \ar@<-1ex>[r]_{1} & G_0 \ar[l]|{1} \ar[u]|{\varepsilon^{G}}}$
	\end{tabular}
\end{center}
\end{example}

\begin{example}
	Let $(A,B,\partial)$ be a topological crossed module. Then $\pi(A,B,\partial)=(\pi(A\rtimes B),\pi(B),A\rtimes B,B)$ becomes a double group-groupoid
\begin{center}
	\begin{tabular}{m{0.05\linewidth}m{0.15\linewidth}}
		\centering $\mathcal{G}$ : &
		\centering $\xymatrix @=4pc {
			\pi(A\rtimes B) \ar@<1ex>[r]^{d_{0}^{h}} \ar@<-1ex>[r]_{d_{1}^{h}} \ar @<1ex>[d]^{d_1^{v}}  \ar@<-1ex>[d]_{d_0^{v}} & \pi(B) \ar[l]|{\varepsilon^h}  \ar @<1ex>[d]^{d_1^{{H}}} \ar @<-1ex> [d]_{d_0^{ H}} \\
			A\rtimes B \ar[u]|{\varepsilon^{v}}  \ar@<1ex>[r]^{d_{0}^{V}} \ar@<-1ex>[r]_{d_{1}^{V}} & B \ar[l]|{\varepsilon^V} \ar[u]|{\varepsilon^{ H}}}$
	\end{tabular}
\end{center}
\end{example}

\begin{example} \label{exampledgpd} Let  $(G, H,\partial)$ be a crossed module over group-groupoids. Then we have a double group-groupoid as follows
\[\xymatrix @=5pc {
			G\rtimes H \ar@<1ex>[r]^{d_{0}^{h}} \ar@<-1ex>[r]_{d_{1}^{h}} \ar @<2.5ex>[d]^{d_1^{  G}\times d_1^{  H}}  \ar@<-2.5ex>[d]_{d_0^{  G}\times d_0^{  H}} & H   \ar[l]|{\varepsilon^h}  \ar @<1ex>[d]^{d_1^{{  H}}} \ar @<-1ex> [d]_{d_0^{  H}} \\
			G_0\rtimes H_0 \ar[u]|{\varepsilon^{v}\times \varepsilon^{  H}}  \ar@<1ex>[r]^{d_{0}^{V}} \ar@<-1ex>[r]_{d_{1}^{V}} & H_0 \ar[l]|{\varepsilon^V} \ar[u]|{\varepsilon^{  H}}}\]
where $d_{0}^{h}(a,b)=b$,  $d_{1}^{h}(a,b)=\partial_1(a)+b$, $\varepsilon^{h}(b)=(0,b)$, $m^h((a_1,b_1),(a,b))=(a_1+a,b)$ for $b_1=\partial_1(a)+b$ and $d_{0}^{V}(x,y)=y$, $d_{1}^{V}(x,y)=\partial_0(x)+y$, $\varepsilon^{V}(y)=(0,y)$, $m^V((x_1,y_1),(x,y))=(x_1+x,y)$ for $y_1=\partial_0(x)+y$. A square $(a,b)$ in $G\rtimes H$ is
\[\xymatrix@=.7pc{
	y \ar[dd]_-{(x,y)} \ar[rr]^{b} & & y_1 \ar[dd]^-{(x_1,y_1)} \\
	& (a,b) & \\
	\partial_0(x)+y \ar[rr]_{\partial_1(a)+b}  & & \partial_0(x_1)+y_1 }\]
for $a\in G(x,x_1)$ and $b\in H(y,y_1)$. If $(a,b)$, $(a_1,b_1)$ and $(a',b')$ are squares in $G\rtimes H$ with $a\in G(x,x_1)$, $a_1\in G(x_1,x_2)$, $a'\in G(x',x_1')$, $b\in H(y,y_1)$, $b_1\in H(y_1,y_2)$ and $b'=\partial_1(a)+b$ then
\begin{align*}
(a_1,b_1)\circ_{h}(a,b)&= (a_1\circ a,b_1\circ b),\\
(a',b')\circ_{v}(a,b)&= (a'+a,b).
\end{align*}
\end{example}

In the following proposition, the items (\ref{p1})-(\ref{p4}) are obtained from the fact that $(d_{0}^{h}$, $d_{0}^{V})$, $(d_{1}^{h}$, $d_{1}^{V})$, $(\varepsilon^{h}$, $\varepsilon^{V})$, $(m^h,m^V)$ and $(n^h,n^V)$ are functors and the item (\ref{p5}) is obtained from the fact that $((S,V),(H,P))$ has an internal groupoid structure within groups.

\begin{proposition}
Let $\mathcal{G}=(S,H,V,P)$ be a double group-groupoid. Then for $i,j\in \{0,1\}$ the following are satisfied:
\begin{enumerate}
	\item\label{p1} $d_{i}^{H}d_{j}^{h}=d_{j}^{V}d_{i}^{v}$ , $\varepsilon^{H}d_{i}^{V}=d_{i}^{h} \varepsilon^{v}$ and  $d_{i}^{h} m^{v}=m^{H}(d_{i}^{h}\times d_{i}^{h})$,
	\item\label{p2} $d_{i}^{v}\varepsilon^h=\varepsilon^V d_{i}^{H}$, $\varepsilon^{v}\varepsilon^V=\varepsilon^h\varepsilon^{H}$ and $m^{v}(\varepsilon^h\times\varepsilon^h)=\varepsilon^h m^{H}$,
	\item\label{p3} $d_{i}^{v}m^h=m^V(d_{i}^{v}\times d_{i}^{v})$, $m^h(\varepsilon^{v}\times\varepsilon^{v})=\varepsilon^{v}m^V$ and  $m^{v}(m^h\times m^h)=m^h(m^{v}\times m^{v})$,
	\item\label{p4} $d_{i}^{v}n^v=n^V d_{i}^{v}$, $\varepsilon^{v}n^V=n^v \varepsilon^{v}$ and  $m^{v}(n^v\times n^v)=n^v m^{v}$,
	\item\label{p5} $d_{i}^{h}\varepsilon_{}^{h}=1_{S}$, $d_{i}^{V}\varepsilon_{}^{V}=1_{V}$, $d_{i}^{h}m_{}^{h}=d_{i}^{h}\pi_{(2-i)}$, $d_{i}^{V}m_{}^{V}=d_{i}^{V}\pi_{(2-i)}$, $m^{h}(1_{S}\times m^{h})=m^{h}(m^{h}\times 1_{S})$, $m^{V}(1_{V}\times m^{V})=m^{V}(m^{V}\times 1_{V})$, $m^{h}(\varepsilon^{h} d_{0}^{h},1_{S})=m^{h}(1_{S},\varepsilon^{h} d_{1}^{h})=1_{S}$, $m^{V}(\varepsilon^{V} d_{0}^{V},1_{V})=m^{V}(1_{V},\varepsilon^{V} d_{1}^{V})=1_{V}$, $m^{h}(1_{S},n^{h})=\varepsilon^{h} d_{0}^{h}$, $m^{V}(1_{V},n^{V})=\varepsilon^{V} d_{0}^{V}$ and $m^{h}(n^{h},1_{S})=\varepsilon^{h} d_{1}^{h}$, $m^{V}(n^{V},1_{V})=\varepsilon^{V} d_{1}^{V}$.
\end{enumerate}
\end{proposition}

%\begin{corollary}
%\begin{enumerate}
% \item By  conditions  (1)-(4)  we have that  $(d_{0}^{h}$, $d_{0}^{V})$, $(d_{1}^{h}$, $d_{1}^{V})$, $(\varepsilon^{h}$, $\varepsilon^{V})$, $(m^h,m^V)$ and $(n^h,n^V)$ are functors, respectively
%\item  By the conditions in (5)  we have that  $\mathcal{G}=(S,H,V,P)$, i.e.
%\[\xymatrix@R=4mm@C=17mm{
%(S,V) {_{(d_{0}^{h},d_{0}^{V})}\times_{(d_{1}^{h},d_{1}^{V})}}  (S,V) \ar[r]^-{(m^h,m^V)} & (S,V) \ar@(dl,dr)[]_{(n^v,n^V)} \ar@<.7ex>[r]^-{(d_{0}^{h},d_{0}^{V})} \ar@<-.4ex>[r]_-{(d_{1}^{h},d_{1}^{V})}  & (H,P)  \ar@(u,u)[l]_-{(\varepsilon^h,\varepsilon^V)}
%}\]
%is equipped with an internal groupoid structure where $(S,V)=(S,V,d_{0}^{v}$, $d_{1}^{v}$, $\varepsilon^{v}$, $n^{v},m^v)$ and $(H,P)=(H,P$, $d_{0}^{H}$, $d_{1}^{H}$, $\varepsilon^{H}$, $n^{H},m^H)$.
%\end{enumerate}\qed
%\end{corollary}

Let $\mathcal{G}$ and $\mathcal{G}'$ be two double group-groupoids. A morphism form  $\mathcal{G}$ to $\mathcal{G}'$ is a double groupoid morphism $\mathcal{F}=(f_s,f_h,f_v,f_p)\colon\mathcal{G}\rightarrow\mathcal{G}'$  such that $f_s\colon S\rightarrow S'$, $f_h\colon H\rightarrow H'$, $f_v\colon V\rightarrow V'$ and $f_p\colon P\rightarrow P'$ are group homomorphisms. Such a morphism of double group-groupoids may be denoted by a diagram as follows:

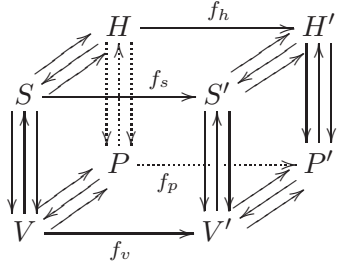
\[\xymatrix@R=4mm@C=7mm{
	& H \ar@{..>}@<.9ex>[dd]\ar@{..>}@<-.9ex>[dd] \ar[dl] \ar[rr]^{f_h} & & H' \ar[dl]\ar@<.9ex>[dd]\ar@<-.9ex>[dd]  \\
	S  \ar[rr]^(.7){f_s} \ar@<.9ex>[ur]\ar@<-.9ex>[ur] \ar@<.9ex>[dd]\ar@<-.9ex>[dd] & & S' \ar@<.9ex>[ur]\ar@<-.9ex>[ur] \ar@<.9ex>[dd]\ar@<-.9ex>[dd] \\
	& P \ar@{..>}'[r]_{f_p}[rr] \ar[dl]\ar@{..>}[uu] & & P' \ar[uu]\ar[dl] \\
	V \ar[uu] \ar@<.9ex>[ur]\ar@<-.9ex>[ur] \ar[rr]_{f_v} & & V' \ar[uu] \ar@<.9ex>[ur]\ar@<-.9ex>[ur] \\
}\]

We write  $\DGG$ for the category with objects double groupoids and morphisms as arrows.

\begin{lemma}\label{lemcomp}
	In a double group-groupoid, as a consequence of interchange laws (Eq. \ref{inthv}), the vertical and horizontal compositions of squares can be written in terms of the group operations as
\begin{align*}
\beta_1\circ_{h}\beta&= \beta_1-\varepsilon^h d_{1}^{h}(\beta)+\beta=\beta-\varepsilon^h d_{1}^{h}(\beta)+\beta_1,\\
\alpha_1\circ_{v}\alpha&= \alpha_1-\varepsilon^{v}d_{1}^{v}(\alpha)+\alpha=\alpha-\varepsilon^{v}d_{1}^{v}(\alpha)+\alpha_1
\end{align*}	
for all $\alpha,\alpha_1,\beta,\beta_1\in G$ such that $d_{1}^{v}(\alpha)=d_{0}^{v}(\alpha_1)$ and $d_{1}^{h}(\beta)=d_{0}^{h}(\beta_1)$.\qed
\end{lemma}
Thus the horizontal inverse of $\beta\in S$ is \[\beta^{-h}=\varepsilon^h d_{0}^{h}(\beta)-\beta+\varepsilon^h d_{1}^{h}(\beta)=\varepsilon^h d_{1}^{h}(\beta)-\beta+\varepsilon^h d_{0}^{h}(\beta)\] and the  vertical inverse of $\alpha\in S$ is \[\alpha^{-v}=\varepsilon^{v}d_{0}^{v}(\alpha)-\alpha+\varepsilon^{v}d_{1}^{v}(\alpha)=\varepsilon^{v}d_{1}^{v}(\alpha)-\alpha+
\varepsilon^{v}d_{0}^{v}(\alpha).\]
In particular, if $\alpha\in\Ker d_{0}^{v}$ and $\beta\in\Ker d_{0}^{h}$ then
\begin{align*}
\beta^{-h}=-\beta+\varepsilon^h d_{1}^{h}(\beta)=\varepsilon^h d_{1}^{h}(\beta)-\beta \tag{4.2}\label{hinv}
\end{align*}
and
\begin{align*}
\alpha^{-v}=-\alpha+\varepsilon^{v}d_{1}^{v}(\alpha)=\varepsilon^{v}d_{1}^{v}(\alpha)-\alpha. \tag{4.3}\label{vinv}
\end{align*}

In Lemma \ref{lemcomp} if we take $\beta$ with $d_{1}^{h}(\beta)=0$ and $\alpha$ with $d_{1}^{v}(\alpha)=0$ then we obtain the following result.

\begin{corollary}\label{kercom}
	Let $\mathcal{G}$ be a double group-groupoid and $\alpha,\alpha_1,\beta,\beta_1\in S$ with $d_{1}^{v}(\alpha)=0=d_{0}^{v}(\alpha_1)$ and $d_{1}^{h}(\beta)=0=d_{0}^{h}(\beta_1)$. Then
\begin{align*}
\beta_1\circ_{h}\beta&= \beta_1+\beta=\beta+\beta_1,\\
\alpha_1\circ_{v}\alpha&= \alpha_1+\alpha=\alpha+\alpha_1
\end{align*}
i.e. squares in $\Ker d_{0}^{v}$ (resp. $\Ker d_{0}^{h}$) and in $\Ker d_{1}^{v}$ (resp. $\Ker d_{1}^{h}$) are commutative.
\end{corollary}

Following corollary is a consequence of Equations (\ref{hinv}),(\ref{vinv}) and Corollary \ref{kercom}.

\begin{corollary}
	Let $\mathcal{G}$ be a double group-groupoid. Then \[\alpha+\alpha_1-\alpha=\varepsilon^{v}d_{1}^{v}(\alpha)+\alpha_1-\varepsilon^{v}d_{1}^{v}(\alpha)\] and \[\beta+\beta_1-\beta=\varepsilon^{h}d_{1}^{h}(\beta)+\beta_1-\varepsilon^{h}d_{1}^{h}(\beta)\] for all $\alpha,\alpha_1\in \Ker d_{0}^{v}$ and $\beta,\beta_1\in \Ker d_{0}^{h}$.
\end{corollary}

\begin{theorem}
	The category $\XMod(\GpGd)$ of crossed modules over group-groupoids is equivalent to the category $\DGG$ of double group-groupoids.
\end{theorem}

\begin{proof}
 Example \ref{exampledgpd} gives rise to a functor  $\theta\colon\XMod(\GpGd)\rightarrow\DGG$
which associates a crossed module over group-groupoids with a double groupoid.

Conversely for a double group-groupoid  $\mathcal{G}=(S,H,V,P)$, we define a crossed module $(K,H,\partial)$ associated with  $\mathcal{G}$  where
\begin{align*}
  K&= (\Ker d_{0}^{h},\Ker d_{0}^{V},d_{0}^{v},d_{1}^{v},\varepsilon^{v},m^{v},n^{v}),\\
  H&= (H,P,d_{0}^{H},d_{1}^{H},\varepsilon^{H},m^{H},n^{H})
\end{align*}
and $\partial=(\partial_1=d_{1}^{h},\partial_0=d_{1}^{V})$. Such a crossed module can be  visualized as in the following  diagram
\[\xymatrix @=4pc {
	\Ker d_{0}^{h} \ar[r]^{\partial_1} \ar @<1ex>[d]^{d_1^{v}}  \ar@<-1ex>[d]_{d_0^{v}} & H  \ar @<1ex>[d]^{d_1^{{  H}}} \ar @<-1ex> [d]_{d_0^{  H}} \\
	\Ker d_{0}^{V} \ar[u]|{\varepsilon^{v}}  \ar[r]^{\partial_0}  & P  \ar[u]|{\varepsilon^{  H}}}\]
Here the action of $H$ on $\Ker d_{0}^{h}$ is given by
\[b\cdot a=\varepsilon^h(b)+a-\varepsilon^h(b)\] for $b\in H$ and $a\in\Ker d_{0}^{h}$. Hence we a functor
 $\gamma\colon \DGG\rightarrow \XMod(\GpGd)$

 We now show that these functors are equivalences of categories. In order to define a natural equivalence $S\colon\theta\gamma\Rightarrow 1_{\DGG}$, for an object $\mathcal{G}$ in $\DGG$ a morphism  $S_{\mathcal{G}}=(f_s,1,f_v,1)\colon\theta\gamma(\mathcal{G})\rightarrow\mathcal{G}$ is defined by $f_v(a,x)=a+\varepsilon^{V}(x)$ and $f_s(\alpha,b)=\alpha-\varepsilon^{h}(b)$ for $(a,x)\in V\rtimes P$ and $(\alpha,b)\in S\rtimes H$.

Conversely, a natural equivalence $T\colon 1_{\XMod(\GpGd)}\Rightarrow\gamma\theta$ can be defined such a way that  for  a crossed module $\mathcal{C}=(  G,  H,\partial)$,  the morphism  $T_{\mathcal{C}}=(f,g)\colon \mathcal{C}\rightarrow \gamma\theta(\mathcal{C})$ is given by  $f_1(a)=(0,a)$ for $a\in G$ and $g$ is the identity.

Other details are straightforward and so are omitted.
\end{proof}

%\subsection{$\XMod(\GpGd) \cong \XSq(\Gp)$}

\section{Crossed modules and crossed squares}
In this section we prove that the category  $\XMod(\GpGd)$ of crossed modules over group-groupoids and the category $\XSq(\Gp)$ of crossed squares over groups  are equivalent.

 We now  recall the definition of crossed squares from \cite{BrLod}. Further we will prove that the category $\XMod(\GpGd)$ of crossed modules over group-groupoids and the category $\XSq(\Gp)$ of crossed squares over groups are equivalent.

\begin{definition}\label{xsq}
A crossed square over groups
\begin{center}
	\begin{tabular}{m{0.05\linewidth} m{0.15\linewidth}}
		\centering $\mathcal{X}$ : &
		\centering $\xymatrix @=3pc {
			L \ar[r]^\lambda \ar[d]_{\lambda'}  & M \ar[d]^{\mu} \\
			N  \ar[r]_{\nu}  & P}$
	\end{tabular}
\end{center}
consists of four morphisms of groups $\lambda\colon L\rightarrow M$, $\lambda'\colon L\rightarrow N$, $\mu\colon M \rightarrow P$ and $\nu\colon N\rightarrow P$, such that $\nu \lambda'= \mu \lambda$ together with actions of the group $P$ on $L$, $M$, $N$ on the left, conventionally, (and hence actions of $M$ on $L$ and $N$ via $\mu$ and of $N$ on $L$ and $M$ via $\nu$) and a function $h\colon M\times N\rightarrow L$. These are subject to the following axioms:
\begin{enumerate}[label={[CS \arabic{*}]}, leftmargin=2cm]
	\item $\lambda$, $\lambda'$ are $P$-equivariant  and $\mu$, $\nu$ and $\kappa=\mu\lambda$ are crossed modules,
	\item $\lambda h(m,n)=m+n\cdot(-m)$,\ \ $\lambda' h(m,n)=m\cdot n-n$,
	\item $h(\lambda(l),n)=l+n\cdot(-l)$,\ \ $h(m,\lambda'(l))=m\cdot l-l$,
	\item $h(m+m',n)=m\cdot h(m',n)+h(m,n)$,\ \ $h(m,n+n')=h(m,n)+n\cdot h(m,n')$,
	\item $h(p\cdot m,p\cdot n)=p\cdot h(m,n)$ and
%	\item $m\cdot(n\cdot l)+h(m,n)=h(m,n)+n\cdot(m\cdot l)$
\end{enumerate}
for all $l\in L$, $m,m'\in M$, $n,n'\in N$ and $p\in P$.
\end{definition}

Norrie \cite{KNor} mentioned that a normal subcrossed module denotes a crossed square as a normal subgroup denotes a crossed module. With this idea a normal subcrossed square will be in form a crossed 3-cube, and so on.

\begin{example}\cite{KNor}
	Let $(A,B,\partial)$ be crossed module and $(S,T,\sigma)$ a normal subcrossed module of $(A,B,\partial)$. Then
	\[\xymatrix @=3pc {
		S \ar@{^{(}->}[d]_\inc \ar[r]^{\sigma}  & T \ar@{^{(}->}[d]^{\inc}  \\
		A  \ar[r]_{\partial}  & B}\]
	forms a crossed square of groups where the action of $B$ on $S$ is induced action from the action of $B$ on $A$ and the action of $B$ on $T$ is conjugation. The h map is defined by
	\[\begin{array}{rccl}
	h: & T\times A & \rightarrow & S \\
	   &(t,a)      & \mapsto     & h(t,a)=t\cdot a -a
	\end{array}\]
	for all $t\in T$ and $a\in A$.
\end{example}

A geometric example of crossed squares is the fundamental crossed square which is defined in \cite{Loday82} as follows: Suppose given a commutative square of spaces
\begin{center}
	\begin{tabular}{m{0.05\linewidth} m{0.15\linewidth}}
		\centering ${\bm X}$ : &
		\centering $\xymatrix @=3pc {
			C \ar[r]^f \ar[d]_{g}  & A \ar[d]^{a} \\
			B  \ar[r]_{b}  & X}$
	\end{tabular}
\end{center}
Let $F(f)$ be the homotopy fibre of $f$ and $F(X)$ the homotopy fibre of $F(g)\rightarrow F(a)$. Then the commutative square of groups
\begin{center}
	\begin{tabular}{m{0.08\linewidth} m{0.15\linewidth}}
		\centering ${\Pi\bm X}$ : &
		\centering $\xymatrix @=3pc {
			\pi_1F(\bm X) \ar[r] \ar[d]  & \pi_1F(g) \ar[d] \\ \pi_1F(f)  \ar[r]  & \pi_1(C)}$
	\end{tabular}
\end{center}
associated to $\bm X$ is naturally equipped with a structure of crossed square. $\Pi\bm X$ is called the fundamental crossed square of $\bm X$ \cite{BrLod}.

\begin{example}
	Let $(A,B,\alpha)$ be a topological crossed module. Then we know that $(\pi A,A,d_{0}^{A},d_{1}^{A},\varepsilon^A,m^A)$ and $(\pi B,B,d_{0}^{B},d_{1}^{B},\varepsilon^B,m^B)$ are fundamental group-groupoids. Then
	\[\xymatrix{
		\Ker d_{0}^{A} \ar[r]^{\pi\alpha} \ar[d]_{d_{1}^{A}}  & \Ker d_{0}^{B} \ar[d]^{d_{1}^{B}}  \\
		A \ar[r]_{\alpha}    & B}\]
	has a crossed square structure where $h([\beta],a)=[\beta\cdot a-a]$ for all $[\beta]\in \Ker d_{0}^{B}$ and $a\in A$. Here the path $(\beta\cdot a-a)\colon [0,1]\rightarrow A$ is given by $(\beta\cdot a-a)(r)=\beta(r)\cdot a-a$ for all $r\in [0,1]$.
\end{example}

A morphism of crossed squares from $\mathcal{X}_1$ to $\mathcal{X}_2$ consists of four group homomorphisms $f_L\colon L_1\rightarrow L_2$, $f_M\colon M_1\rightarrow M_2$, $f_N\colon N_1\rightarrow N_2$ and $f_P\colon P_1\rightarrow P_2$ which are compatible with the actions and the functions $H$ and $h_2$.

%\[\xymatrix@R=7mm@C=10mm{
%L_1 \ar[rr]^{\lambda_1} \ar[d]_{\lambda'_1} \ar@{..>}'[dr]|-{f_L}[ddrr] &  & %M_1  \ar[d]^{\mu_1} \ar@{..>}[ddrr]|-{f_M} \\ N_1  \ar[rr]|{\nu_1} %\ar@{..>}[ddrr]|-{f_N} &  & P_1 \ar@{..>}'[dr]|-{f_P}[ddrr] \\
%& & L_2 \ar[rr]|{\lambda_2} \ar[d]_{\lambda'_2}  &  & M_2 \ar[d]^{\mu_2} \\ %& & N_2  \ar[rr]_{\nu_2} &  & P_2
%}\]

\[\xymatrix@R=4mm@C=7mm{
& M_1 \ar@{..>}'[d][dd]_{\mu_1} \ar[rr]^{f_M} & & M_2 \ar[dd]^{\mu_2}  \\
L_1 \ar[dd]_{\lambda'_1} \ar[rr]^(.7){f_L} \ar[ur]^{\lambda_1} & & L_2 \ar[dd]_(.3){\lambda'_2} \ar[ur]_{\lambda_2} \\
& P_1 \ar@{..>}'[r]_{f_P}[rr] & & P_2 \\
N_1 \ar[rr]_{f_N} \ar[ur]^{\nu_1} & & N_2 \ar[ur]_{\nu_2} \\
}\]

Category of crossed squares over groups with morphisms between crossed squares defined above is denoted by $\XSq(\Gp)$. Crossed squares are equivalent to the crossed modules over crossed modules.

\begin{theorem}
	The category $\XMod(\GpGd)$ of crossed modules over group-groupoids is equivalent to the category $\XSq(\Gp)$ of crossed squares over groups.
\end{theorem}

\begin{proof}
Define a functor $\delta\colon\XMod(\GpGd)\rightarrow\XSq(\Gp)$ as in the following way: Let $(G,  H,\partial)$ be a crossed module over group-groupoids. If we set $L=\Ker d_{0}^{  G}$, $M=\Ker d_{0}^{  H}$, $N=G_0$, $P=H_0$, $\lambda=\partial_1$, $\lambda'=d_{1}^{  G}$, $\mu=d_{1}^{  H}$ and $\nu=\partial_0$ then
\begin{center}
\begin{tabular}{m{0.19\linewidth}m{0.15\linewidth}}
	\centering $\delta((  G,  H,\partial))=\mathcal{X}$ :  &
	\centering $\xymatrix @=4pc {
		L \ar[r]^{\lambda} \ar[d]_{\lambda'} & M   \ar[d]^{\mu} \\
		N \ar[r]_{\nu} & P }$
\end{tabular}
\end{center}
is a crossed square. Here the action of $P$ on $N$ is already given and the action of $M$ on $L$ is induced action. Actions of $P$ on $M$ and on $L$ are given by
\[
p\cdot m = \varepsilon^{  H}(p)+m-\varepsilon^{  H}(p) \ \ \ \text{ and } \ \ \ p\cdot l = \varepsilon^{  H}(p)\cdot l
\]
respectively, for $p\in P$, $m\in M$ and $l\in L$. Moreover,
\[\begin{array}{rcccl}
h & \colon & M\times N & \rightarrow & L \\
& & (m,n)      & \mapsto     & h(m,n)=m\cdot \varepsilon^{G}(n)-\varepsilon^{G}(n).
\end{array}\]
Now we will verify that $\mathcal{X}$ satisfies the conditions \textbf{[CS1]-[CS6]} of Definition \ref{xsq}. Let $l\in L$, $m,m'\in M$, $n,n'\in N$ and $p\in P$. Then
\begin{enumerate}[label=\textbf{[CS \arabic{*}]}, leftmargin=2cm]
	\item It is easy to see that $\mu$, $\nu$ and $\kappa=\mu\lambda$ are crossed modules. Now we will show that $\lambda$ and $\lambda'$ are $P$-equivariant. Let $p\in P$ and $l\in L$. Then
\[\lambda(p\cdot l)=\partial_1(\varepsilon^{  H}(p)\cdot l)=\varepsilon^{  H}(p)+\partial_1(l)-\varepsilon^{  H}(p)=p\cdot \lambda(l)\] and \[\lambda'(p\cdot l)=d_{1}^{  G}(\varepsilon^{  H}(p)\cdot l)=d_{1}^{  H}\varepsilon^{  H}(p)\cdot d_{1}^{  G}(l)=p\cdot \lambda'(l)\]
	\item Let $m\in M$ and $n\in N$. Then
	\begin{align*}	
	\lambda h(m,n)&=\partial_1(m\cdot \varepsilon^{G}(n)-  \varepsilon^{G}(n))\\
	&=\partial_1(m\cdot   \varepsilon^{G}(n))-\partial_1(  \varepsilon^{G}(n))\\
	&=m+\partial_1(  \varepsilon^{G}(n))-m-\partial_1(  \varepsilon^{G}(n))\\
	&=m+n\cdot (-m)
	\end{align*}
	and
	\begin{align*}
	\lambda' h(m,n)&=d_{1}^{  G}(m\cdot   \varepsilon^{G}(n)-  \varepsilon^{G}(n))\\
	&=d_{1}^{  G}(m\cdot   \varepsilon^{G}(n))-d_{1}^{  G}(  \varepsilon^{G}(n))\\
	&=d_{1}^{  H}(m)\cdot d_{1}^{  G}(  \varepsilon^{G}(n))-d_{1}^{  G}(  \varepsilon^{G}(n))\\
	&=d_{1}^{  H}(m)\cdot n-n\\
	&=m\cdot n - n
	\end{align*}
	\item Let $l\in L$ and $n\in N$. Then
	\begin{align*}	
	h(\lambda(l),n)&=\partial_1(l)\cdot   \varepsilon^{G}(n)-  \varepsilon^{G}(n) \\
	&=l+  \varepsilon^{G}(n)-l-  \varepsilon^{G}(n) \\
	&=l+n\cdot(-l)
	\end{align*}
	and
	\begin{align*}
	h(m,\lambda'(l))&=m\cdot   \varepsilon^{G}(d_{1}^{  G}(l))-  \varepsilon^{G}(d_{1}^{  G}(l))\\
	&=(m\cdot   \varepsilon^{G}(d_{1}^{  G}(l))-  \varepsilon^{G}(d_{1}^{  G}(l))+l)-l\\
	&=(m\cdot   \varepsilon^{G}(d_{1}^{  G}(l))\circ l)-l\\
	&=((m\cdot   \varepsilon^{G}(d_{1}^{  G}(l)))\circ (1_0\cdot l))-l\\
	&=((m\circ 1_0)\cdot (  \varepsilon^{G}(d_{1}^{  G}(l))\circ l))-l\\
	&=m\cdot l-l
	\end{align*}
	\item Let $m,m'\in M$ and $n\in N$. Then
	\begin{align*}	
	h(m+m',n)&=(m+m')\cdot   \varepsilon^{G}(n)-  \varepsilon^{G}(n) \\
	&=m\cdot (m'\cdot   \varepsilon^{G}(n))-  \varepsilon^{G}(n) \\
	&=m\cdot (m'\cdot   \varepsilon^{G}(n))+m\cdot (-  \varepsilon^{G}(n)+  \varepsilon^{G}(n))-  \varepsilon^{G}(n) \\
	&= m\cdot(m'\cdot   \varepsilon^{G}(n)-  \varepsilon^{G}(n))+m\cdot   \varepsilon^{G}(n) -  \varepsilon^{G}(n) \\
	&= m\cdot h(m',n)+h(m,n)
	\end{align*}
	and
	\begin{align*}
	h(m,n+n')&= m\cdot   \varepsilon^{G}(n+n')-  \varepsilon^{G}(n+n') \\
	&= m\cdot (  \varepsilon^{G}(n)+  \varepsilon^{G}(n'))-  \varepsilon^{G}(n')-  \varepsilon^{G}(n) \\
	&= (m\cdot   \varepsilon^{G}(n)-  \varepsilon^{G}(n))+  \varepsilon^{G}(n)+(m\cdot  \varepsilon^{G}(n')-  \varepsilon^{G}(n'))-  \varepsilon^{G}(n)\\
	&=h(m,n)+n\cdot h(m,n')
	\end{align*}
	\item Let $m\in M$, $n\in N$ and $p\in P$. Then
	\begin{align*}
	h(p\cdot m,p\cdot n)&= h(\varepsilon^{  H}(p)+m-\varepsilon^{  H}(p),p\cdot n) \\
	&= (\varepsilon^{  H}(p)+m-\varepsilon^{  H}(p))\cdot   \varepsilon^{G}(p\cdot n)-  \varepsilon^{G}(p\cdot n)\\
	&= (\varepsilon^{  H}(p)+m-\varepsilon^{  H}(p))\cdot (\varepsilon^{  H}(p)\cdot  \varepsilon^{G}(n))-(\varepsilon^{  H}(p)\cdot  \varepsilon^{G}(n))\\
	&= (\varepsilon^{  H}(p)+m)\cdot   \varepsilon^{G}(n)-(\varepsilon^{  H}(p)\cdot  \varepsilon^{G}(n))\\
	&= \varepsilon^{  H}(p)\cdot (m\cdot   \varepsilon^{G}(n))+\varepsilon^{  H}(p)\cdot(-  \varepsilon^{G}(n))\\
	&= \varepsilon^{  H}(p)\cdot (m\cdot   \varepsilon^{G}(n)-  \varepsilon^{G}(n))\\
	&=p\cdot h(m,n)
	\end{align*}
\end{enumerate}
Now let  $(f=(f_1,f_0),g=(G,g_0))\colon (  G,  H,\partial)\rightarrow(  G',  H',\partial')$ be a morphism in $\XMod(\GpGd)$ then $\delta_1(f,g)=(f_L=f_1,f_M=G,f_N=f_0,f_P=g_0)$ is a morphism of crossed modules over group-groupoids.

Conversely, define a functor $\eta\colon\XSq(\Gp)\rightarrow\XMod(\GpGd)$ by following way: Let $\mathcal{X}=(L,N,M,P)$ be a crossed square over groups. Let $  G$ and $  H$ be the corresponding group-groupoids to crossed modules $(L,N,\lambda')$ and $(M,P,\mu)$, respectively. That is $G=L\rtimes N$, $G_0=N$, $H=M\rtimes P$ and $H_0=P$. Moreover, $\partial_1=\lambda\times\nu$ and $\partial_0=\nu$ whilst the action of $H$ on $G$ is given by
\[\begin{array}{ccl}
 \left( M\rtimes P \right) \times \left( L\rtimes N \right)  & \rightarrow & \left( L\rtimes N \right) \\
\left( (m,p), (l.n) \right)      & \mapsto     & (m,p)\cdot(l,n)=\left(m\cdot(p\cdot l)+h(m,p\cdot n), p\cdot n \right).
\end{array}\]
Here $(L\rtimes N,M\rtimes P,\lambda\times\nu)$ becomes a crossed module with the action given above. This crossed module is called the semi-direct product of crossed modules. Hence $\eta_0(\mathcal{X})=(  G,  H,\partial)$ is a crossed module over group-groupoids. Now, let $(f_L,f_M,f_N,f_P)$ be a morphism in $\XSq(\Gp)$. Then $\eta_1(f_L,f_M,f_N,f_P)=(f=(f_L\times f_N,f_N),g=(f_M\times f_P,f_P))$ is morphism in $\XMod(\GpGd)$.

A natural equivalence $S\colon \eta\delta\Rightarrow 1_{\XMod(\GpGd)}$ is given by the map $S_{(  G,  H,\partial)}\colon\eta\delta(  G,  H,\partial)\rightarrow (  G,  H,\partial)$ where $S_{(  G,  H,\partial)}$ is identity on both $G_0$ and $H_0$, and defined by $a\mapsto (a-  \varepsilon^{G}(d_{1}^{  G}(a)),d_{1}^{  G}(a))$ on $G$ and $b\mapsto (b-\varepsilon^{  H}(d_{1}^{  H}(b)),d_{1}^{  H}(b))$ on $H$.

Conversely, for a crossed square over groups $\mathcal{X}$, a natural equivalence $T\colon 1_{\XSq(\Gp)}\Rightarrow \delta\eta$ is given by the map $T_{\mathcal{X}}\colon \mathcal{X}\rightarrow \delta\eta(\mathcal{X})$ where $T_{\mathcal{X}}$ is identity on $P$ and $M$ and and defined by $m\mapsto (m,0)$ and $l\mapsto (l,0)$.

Other details are straightforward and so is omitted.
\end{proof}

\section{Conclusion}

Results obtained in this paper can be given in a more generic cases such as for an arbitrary category of groups with operations, or for an arbitrary modified category of interest etc. Moreover, notions of lifting \cite{Mu-Sa}, covering \cite{Br-Mu1} and actor crossed module \cite{KNor} and homotopy of crossed module morphisms (particularly derivations) can be interpreted in the categories mentioned in this manuscript.

\small{

}

\end{document}